\documentclass[a4paper,11pt]{article}
\usepackage{amsmath, amsfonts, amscd, amssymb, amsthm, enumerate, xypic}

\newcommand{\Ker}{\operatorname{Ker}}

\newcommand{\mrk}{\operatorname{mrk}}

\newcommand{\rk}{\operatorname{rk}}

\renewcommand{\setminus}{\smallsetminus}

\def\K{\mathbb{K}}

\def\calL{\mathcal{L}}

\def\calN{\mathcal{N}}

\def\calR{\mathcal{R}}
\def\calS{\mathcal{S}}
\def\calT{\mathcal{T}}

\theoremstyle{definition}

\theoremstyle{plain}
\newtheorem{theo}{Theorem}[section]
\newtheorem{prop}[theo]{Proposition}

\newtheorem{claim}{Claim}

\theoremstyle{plain}

\theoremstyle{remark}
\newtheorem{Rems}{Remarks}[section]
\newtheorem{Rem}[Rems]{Remark}

\title{On the minimal rank in non-reflexive operator spaces over finite fields}
\author{Cl\'ement de Seguins Pazzis\footnote{Universit\'e de Versailles Saint-Quentin-en-Yvelines, Laboratoire de Math\'ematiques
de Versailles, 45 avenue des Etats-Unis, 78035 Versailles cedex, France}
\footnote{e-mail address: dsp.prof@gmail.com}}

\begin{document}

\thispagestyle{plain}

\maketitle

\begin{abstract}
Let $U$ and $V$ be vector spaces over a field $\K$, and
$\calS$ be an $n$-dimensional linear subspace of $\calL(U,V)$.
The space $\calS$ is called algebraically reflexive whenever it contains every linear map $g : U \rightarrow V$
such that, for all $x \in U$, there exists $f \in \calS$ with $g(x)=f(x)$.
A theorem of Meshulam and \v Semrl states that if $\calS$ is not algebraically reflexive
then it contains a non-zero operator $f$ of rank at most $2n-2$, provided that $\K$ has more than $n+2$ elements.
In this article, we prove that the provision on the cardinality of the underlying field is unnecessary.
To do so, we demonstrate that the above result holds for all finite fields.
\end{abstract}



\section{Introduction}

Let $\K$ be an arbitrary field and $U$ and $V$ be vector spaces over $\K$.
Given a linear subspace $\calS$ of the space $\calL(U,V)$ of all linear maps from $U$ to $V$, its reflexive closure is defined as
$$\calR(\calS):=\bigl\{g \in \calL(U,V) : \; \forall x \in U, \; \exists f \in \calS : \; g(x)=f(x)\bigr\};$$
it is obviously a linear subspace of $\calL(U,V)$ that contains $\calS$, and one checks that $\calR(\calS)=\calR(\calR(\calS))$.
One says that $\calS$ is \textbf{(algebraically) reflexive} whenever $\calR(\calS)=\calS$.

An active research topic consists in finding sufficient conditions for the reflexivity of an operator space
in terms of the dimension of $\calS$ and the rank of its elements.
Denote by
$$\mrk(\calS):=\min\bigl\{\rk(f) \mid f \in \calS \setminus \{0\}\bigr\}$$
the minimal rank among the non-zero operators in $\calS$ (here we do not distinguish between infinite cardinals
and simply write $\rk(f)=+\infty$ if $f$ is not a finite rank operator).

Assume now that $\calS$ is non-reflexive and finite-dimensional.
In \cite{Larson},
Larson showed that
$$\mrk(\calS) <+\infty.$$
Hence, a finite-dimensional operator space that contains no non-zero operator of finite rank is always reflexive.
A natural improvement is to give an upper-bound for $\mrk(\calS)$ with respect to the dimension of $\calS$.
In \cite{Ding}, Ding showed that
$$\mrk(\calS) \leq (\dim \calS)^2.$$
Later, this upper-bound was substantially improved by
Meshulam and \v Semrl: in \cite{MeshulamSemrlLAA}, they showed that
$$\# \K>\dim \calS+2 \; \Rightarrow \; \mrk(\calS) \leq 2\dim \calS -2.$$
Earlier, this result had been obtained by Li and Pan for the field of complex numbers \cite{LiPan}.

For $2$-dimensional spaces, this upper bound is known to be optimal (see \cite{MeshulamSemrlLAA}).
For algebraically closed fields, Meshulam and \v Semrl further improved the upper-bound as follows in \cite{MeshulamSemrlPAMS}:
$$\mrk(\calS) \leq \dim \calS.$$
In \cite{dSPLLD1}, we examined whether the upper-bound $2\dim \calS-2$ from Meshulam and \v Semrl's result
was optimal or if one could improve it in the case when $\dim \calS \geq 3$.
First, it was proved that this upper-bound still held under the milder cardinality assumption
$\# \K>\dim \calS$, and then, under that provision, a classification of the non-reflexive $n$-dimensional operator spaces
$\calS$ such that $\mrk \calS=2n-2$ was achieved (see Theorem 6.1 of \cite{dSPLLD1}):
it was shown in particular that the existence of such spaces is connected to the existence of exotic
division algebra structures over the field $\K$, called left-division-bilinearizable (LDB) division algebras.
The existence of LDB division algebras over $\K$ is deeply connected to the quadratic structure of $\K$.
LDB division algebras were entirely classified in \cite{dSPLDB}, and as a consequence the following result was obtained\footnote{When $U$ and $V$ are finite-dimensional, Theorem \ref{betterbound}
is a straightforward consequence of Theorem 6.1 of \cite{dSPLLD1} and of Corollary 1.3 of \cite{dSPLDB}.
To obtain the general case, it suffices to extend the former to all vector spaces $U$ and $V$,
which can be done by noticing that Meshulam and \v Semrl's Corollary 2.5 of \cite{MeshulamSemrlLAA}
states that if $\mrk(\calS)=2\dim \calS-2$ then all the operators in $\calS$ have finite-rank
(provided that $\# \K>n+2$, but it has been shown in \cite{dSPLLD1} that it suffices to assume that $\# \K>n$)
and one can then simply apply the above results to the reduced
space associated with $\calS$, whose source and target spaces are finite-dimensional
(see Section \ref{infinitesection} for the definition of that reduced space).}:

\begin{theo}\label{betterbound}
Let $\calS$ be a non-reflexive $n$-dimensional subspace of $\calL(U,V)$, with
$\# \K>n \geq 3$.
If $\K$ has characteristic not $2$ and $n \not\in \{3,5,9\}$, then
$$\mrk(\calS) \leq 2n-3.$$
If $\K$ has characteristic $2$ and $n-1$ is not a power of $2$, then
$$\mrk(\calS) \leq 2n-3.$$
\end{theo}

For finite fields, this can even be improved as follows:

\begin{theo}\label{finitecase}
Let $\calS$ be a non-reflexive $n$-dimensional subspace of $\calL(U,V)$, with $\K$ a finite field such that
$\# \K>n \geq 3$.
Then,
$$\mrk(\calS) \leq 2n-3.$$
\end{theo}

This follows from Theorem 6.1 of \cite{dSPLLD1} and from the fact, over a finite field, a quadratic form
whose dimension is greater than $2$ is always isotropic.

In this article, we consider the situation of small finite fields.
Until now, the best known result over such fields was the following one:

\begin{prop}[See Theorem 4.5 in \cite{dSPLLD1}]
Let $\calS$ be an $n$-dimensional non-reflexive operator space.
Then,
$$\mrk(\calS) \leq \dfrac{n(n+1)}{2}\cdot$$
\end{prop}

Here, we shall improve this upper-bound as follows, thus generalizing Meshulam and \v Semrl's theorem to all fields:

\begin{theo}\label{maintheo}
Let $\calS$ be an $n$-dimensional non-reflexive operator space.
Then,
$$\mrk(\calS) \leq 2n-2.$$
\end{theo}

To achieve this, we will prove that Theorem \ref{maintheo} holds for all finite fields.
In the case when the source space of the operators in $\calS$ is finite-dimensional,
we will use counting techniques together with very basic results from linear algebra to
obtain the above result (Section \ref{finitedimsection}). These methods were inspired by an article of Meshulam and \v Semrl \cite{MeshulamSemrlPJM}, in which
a similar technique was used to study locally linearly dependent spaces of operators over finite fields.
In the last section, the general result will be derived from this situation by using a theorem of Larson \cite{Larson}.

Before we proceed with the proof of Theorem \ref{maintheo}, we would like to make a few observations.
First of all, Theorem \ref{finitecase} shows that $2n-2$ is not an optimal upper-bound for finite fields of large cardinality and $n \geq 3$.
We do not know whether the upper bound $2n-3$ holds for arbitrary finite fields when $n \geq 3$.
In any case, it is known that the optimal upper-bound must be greater than or equal to $n$, owing to the existence of
$n$-dimensional division algebras over any finite field (see \cite{MeshulamSemrlPAMS}).
Our last remark is that results on non-reflexive spaces are often obtained as special cases of results on locally linearly dependent operator spaces.
Recall that the subspace $\calS \subset \calL(U,V)$ is called \textbf{locally linearly dependent} (in short: \textbf{LLD}) whenever
every vector $x \in U$ is annihilated by some operator $f \in \calS \setminus \{0\}$.
If $\calS$ is non-reflexive and one chooses $g \in \calR(\calS) \setminus \calS$, then
$\calS \oplus \K g$ is an LLD space of which $\calS$ is a linear hyperplane.
Most of the results we have cited are actually special cases of results on linear hyperplanes of LLD spaces
(provided that $\dim \calS \geq 2$; on the other hand, every $1$-dimensional operator space is reflexive).
However, it is not true that a linear hyperplane of an LLD space is always non-reflexive.
One special feature of the proof of Theorem \ref{finitecase} is that we shall use the full power of the non-reflexivity
assumption instead of relying only upon local linear dependence. We do not know whether the upper bound
$2n-2$ in Theorem \ref{maintheo} holds for linear hyperplanes of locally linearly dependent spaces as well (with $n \geq 2$).

\section{Proof of Theorem \ref{maintheo} in the finite-dimensional setting over a finite field}\label{finitedimsection}

Throughout this section, we assume that the field $\K$ is finite and $q$ denotes its cardinality.
Let $U$ and $V$ be vector spaces over $\K$, and $\calS$ be a finite-dimensional non-reflexive subspace of $\calL(U,V)$.
Assume that $U$ is finite-dimensional, set 
$$p:=\dim U \quad \text{and} \quad n:=\dim \calS,$$
and assume that
$$\mrk (\calS)>2n-2,$$
so that
$$p \geq 2n-1.$$
We seek to find a contradiction.
Classically, if $n \leq 1$ then $\calS$ would be reflexive and hence
$$n \geq 2.$$
Let us choose an operator $g \in \calR(\calS) \setminus \calS$.
Then, $$\calT:=g+\calS$$
is an $n$-dimensional affine subspace of $\calL(U,V)$ that does not contain $0$, and the assumption $g \in \calR(\calS)$
translates into:
$$\forall x \in U, \; \exists h \in \calT : \; h(x)=0.$$
Let us consider the set
$$\calN:=\bigl\{(x,h)\in U \times \calT : \; h(x)=0\bigr\}.$$

\begin{claim}\label{claim1}
The affine space $\calT$ contains an operator $h$ such that $\rk h \leq n-1$.
\end{claim}

\begin{proof}
Assume on the contrary that no such operator exists.
For all $x \in U$, either $x=0$ and then all the operators $h \in \calT$ satisfy $h(x)=0$,
or $x \neq 0$ and then at least one operator $h \in \calT$ satisfies $h(x)=0$.
This leads to
$$q^n+q^p-1 \leq \# \calN.$$
On the other hand, for every $f \in \calT$, we have $\dim \Ker f \leq p-n$ and hence
at most $q^{p-n}$ vectors of $U$ are annihilated by $f$. This yields
$$\# \calN \leq q^n\,q^{p-n}=q^p.$$
Combining the above two inequalities leads to $q^n-1 \leq 0$, contradicting $n > 0$.
\end{proof}

Now, let us set
$$r:=\min \bigl\{\rk h \mid h \in \calT\bigr\}.$$
Note that we have just proved that
$$r \leq n-1.$$
In the rest of the proof, we shall use the following simple remark:
if there are distinct operators $h_1$ and $h_2$ in $\calT$ such that $\rk h_1=r$ and $\rk h_2\leq 2n-2-r$, then
$h_1-h_2$ is a non-zero operator of $\calS$ and
$$\rk(h_1-h_2) \leq r+(2n-2-r)=2n-2,$$
which contradicts our assumption that $\mrk(\calS)>2n-2$. Thus, we obtain:

\begin{claim}
The space $\calT$ contains exactly one rank $r$ operator, and all the other ones have their rank greater than $2n-2-r$.
\end{claim}

Next, we prove:

\begin{claim}
One has $r=n-1$.
\end{claim}

\begin{proof}
Recall from the proof of Claim \ref{claim1} that $q^n+q^p-1 \leq \# \calN $.
On the other hand, the sole rank $r$ operator of $\calT$ annihilates exactly $q^{p-r}$ vectors of $U$,
whereas every other operator in $\calT$ annihilates at most $q^{p-2n+1+r}$ vectors.
This leads to
$$\# \calN \leq q^{p-r}+(q^n-1)q^{p-2n+1+r},$$
and hence
$$q^{p-r}(q^r-1) \leq (q^n-1)(q^{p-2n+1+r}-1).$$

In particular, as $r>0$ we find $p-2n+1+r>0$, and factoring yields
$$q^{r-n+1} \geq \frac{1-q^{-r}}{(1-q^{-n})(1-q^{2n-p-1-r})}\cdot$$
Obviously, as $q \geq 2$ and $r>0$, 
$$\frac{1-q^{-r}}{(1-q^{-n})(1-q^{2n-p-1-r})}> 1-q^{-r}\geq \frac{1}{2},$$
and hence $r-n+1 >-1$, which leads to $r \geq n-1$.
\end{proof}

Now, we know that $\calT$ contains one rank $n-1$ operator, which we denote by $h_0$, and all the other ones have greater rank.
Set
$$m:=\# \bigl\{f \in \calT : \; \rk(f) \leq n\bigr\},$$
so that $m \leq q^n$ and $\calT$ contains exactly $m-1$ rank $n$ operators, and
exactly $q^n-m$ operators with rank greater than $n$.
This leads to
\begin{equation}\label{majo3}
\# \calN \leq q^{p-n+1}+(m-1)q^{p-n}+(q^n-m)q^{p-n-1}.
\end{equation}
For every $h \in \calT$ such that $\rk h=n$, we have
$$\dim (\Ker h \cap \Ker h_0) \geq \dim \Ker h+\dim \Ker h_0-\dim U = p-2n+1.$$
Thus, at least $q^{p-2n+1}$ vectors of $\Ker h_0$ belong to $\Ker h$.
Considering the subset
$$\calN':=\calN \cap \bigl((\Ker h_0 \setminus \{0\}) \times (\calT \setminus \{h_0\})\bigr),$$
this leads to
$$\# \calN' \geq (q^{p-2n+1}-1)\,(m-1),$$
and hence
\begin{equation}\label{mino3}
q^n+q^p-1+(q^{p-2n+1}-1)\,(m-1) \leq \# \calN.
\end{equation}
Combining \eqref{majo3} with \eqref{mino3} leads to
\begin{equation}\label{minequality}
q^n+q^p-q^{p-2n+1}-q^{p-n+1}+q^{p-n}-q^{p-1}
\leq m\bigl(q^{p-n}-q^{p-n-1}-q^{p-2n+1}+1\bigr).
\end{equation}
As $q \geq 2$ we have on the other hand
$$q^{p-n}-q^{p-n-1}-q^{p-2n+1}+1 \geq q^{p-n-1}(q-1)-q^{p-2n+1}\geq q^{p-n-1}-q^{p-2n+1} \geq 0,$$
where the last inequality comes from $n \geq 2$.
As $m \leq q^n$ we deduce that
$$q^n+q^p-q^{p-2n+1}-q^{p-n+1}+q^{p-n}-q^{p-1}
\leq q^n\bigl(q^{p-n}-q^{p-n-1}-q^{p-2n+1}+1\bigr).$$
Expanding and simplifying leads to
$$q^{p-n} \leq q^{p-2n+1}.$$
Yet, $q^{p-2n+1}<q^{p-n}$ since $n \geq 2$.

This final contradiction shows that our initial assumption was wrong. This yields
$$\mrk(\calS) \leq 2n-2,$$
thereby completing the proof of Theorem \ref{maintheo} in the special case when $\K$ is finite and the source space of $\calS$ is finite-dimensional.

\section{The generalization to operator spaces between infinite-dimensional spaces}\label{infinitesection}

Now, we complete the proof of Theorem \ref{maintheo} for finite fields.
Assume that $\K$ is finite.

We lose no generality in assuming that $\calS$ is a minimal non-reflexive space.
Then, by a theorem of Larson \cite[Corollary 2.8]{Larson}, all the operators in $\calS$ have finite rank. It follows that
$$U_0:=\underset{f \in \calS}{\bigcap} \Ker f$$
has finite codimension in $U$. Then, every $f \in \calR(\calS)$ naturally induces
a linear operator
$$\overline{f} : U/U_0 \rightarrow V$$
with the same rank as $f$, to the effect that the \textbf{reduced space}
$$\overline{\calS}:=\bigl\{\overline{f} \mid f \in \calS\bigr\}$$
has dimension $n$ and the vector space $\overline{\calR(\calS)}$ is isomorphic to $\calR(\calS)$, whose dimension is greater than $n$.
One checks that $\overline{\calR(\calS)} \subset \calR(\overline{\calS})$ (actually, those spaces are equal),
and hence $\overline{\calS}$ is non-reflexive.
Then, as $U/U_0$ is finite-dimensional, we deduce from Section \ref{finitedimsection} that
$$\mrk (\calS)=\mrk (\overline{\calS}) \leq 2n-2,$$
which completes the proof.

\begin{Rem}
If $\calS$ contains an operator with infinite rank,
then applying the above result to the subspace
$\calS_F$ of all finite rank operators in $\calS$ - which, by a theorem of Larson \cite{Larson}, is non-reflexive - yields $\mrk \calS = \mrk \calS_F \leq 2n-4$.
Therefore, if $\mrk \calS \geq 2n-3$ then $\calS$ contains only finite rank operators.
\end{Rem}


\begin{thebibliography}{1}

\bibitem{Ding}
L. Ding, \emph{On a pattern of reflexivity of operator spaces},
\newblock{Proc. Amer. Math. Soc.}
\newblock{\textbf{124}}
\newblock{(1996)}
\newblock{3101--3108.}

\bibitem{Larson}
D. A. Larson, \emph{Reflexivity, algebraic reflexivity and linear interpolation},
\newblock{Amer. J. Math.}
\newblock{\textbf{283}}
\newblock{(1988)}
\newblock{283--299.}

\bibitem{LiPan}
J. Li, Z. Pan, \emph{Reflexivity and hyperreflexivity of operator spaces},
\newblock{J. Math. Anal. Appl.}
\newblock{\textbf{279}}
\newblock{(2003)}
\newblock{210--215.}

\bibitem{MeshulamSemrlPJM}
R. Meshulam, P. \v Semrl, \emph{Locally linearly dependent operators},
\newblock{Pac. J. Math.}
\newblock{\textbf{203-2}}
\newblock{(2002)}
\newblock{441-459.}

\bibitem{MeshulamSemrlLAA}
R. Meshulam, P. \v Semrl, \emph{Locally linearly dependent operators and reflexivity of operator spaces},
\newblock{Linear Algebra Appl.}
\newblock{\textbf{383}}
\newblock{(2004)}
\newblock{143--150.}

\bibitem{MeshulamSemrlPAMS}
R. Meshulam, P. \v Semrl, \emph{Minimal rank and reflexivity of operator spaces},
\newblock{Proc. Amer. Math. Soc.}
\newblock{\textbf{135}}
\newblock{(2007)}
\newblock{1839--1842.}


\bibitem{dSPLDB}
C. de Seguins Pazzis, \emph{LDB division algebras},
\newblock{preprint,}
\newblock{2013,}
\newblock{arXiv: http://arxiv.org/abs/1312.7800}

\bibitem{dSPLLD1}
C. de Seguins Pazzis, \emph{Local linear dependence seen through duality I},
\newblock{preprint,}
\newblock{2013,}
\newblock{arXiv: http://arxiv.org/abs/1306.1845}



\end{thebibliography}
\end{document}